\documentclass[12pt,reqno]{amsart}
\usepackage{amsmath, amsthm, amssymb}

\usepackage{thmtools}
\usepackage{mathtools} 
\usepackage{dsfont}
\usepackage{color}
\usepackage{enumitem}
\usepackage{hyperref}
\usepackage[inner=1.0in,outer=1.0in,bottom=1.0in, top=1.0in]{geometry}

\newtheorem{theorem}{Theorem}

\newtheorem{lemma}[theorem]{Lemma}
\newtheorem{corollary}[theorem]{Corollary}
\newtheorem{proposition}[theorem]{Proposition}
\newtheorem{conjecture}[theorem]{Conjecture}

\renewcommand{\(}{\left(}
\renewcommand{\)}{\right)}

\def\SL{\mathrm{SL}}

\newcommand\be{\begin{equation}}
\newcommand\ee{\end{equation}}
\newcommand\bee{\begin{equation*}}
\newcommand\eee{\end{equation*}}
\newcommand{\defeq}{\vcentcolon=}

\title{Irreducibility of the zero polynomials of Eisenstein series}
\author{Oscar E. Gonz\'alez}
\address{Department of Mathematics, University of Illinois at Urbana-Champaign, Urbana, IL 61801}
\email{oscareg2@illinois.edu}
\begin{document}

\begin{abstract}
Let $E_k$ be the normalized Eisenstein series of weight $k$ on $\mathrm{SL}_{2}(\mathbb{Z})$. Let $\varphi_k$ be the polynomial that encodes the $j$-invariants of non-elliptic zeros of $E_k$. In 2001, Gekeler observed that the polynomials $\varphi_k$ seem to be irreducible (and verified this claim for $k\leq 446$). We show that $\varphi_k$ is irreducible for infinitely many $k$.
\end{abstract}
\maketitle

\section{Introduction}
Let $E_k$ be the Eisenstein series of even weight $k\geq 2$ on 
$\SL_{2}(\mathbb{Z})$
given by
\be
 E_k(z) \defeq 1-\frac{2k}{B_k} \sum_{n=1}^{\infty} \sigma_{k-1}(n) e^{2 \pi i n z},
\ee
and 
$\Delta \defeq \frac{1}{1728}(E_4^3-E_6^2)$.
Let $j \defeq \frac{E_4^3}{\Delta}$ and
define the polynomial $\varphi_k(X)$ by
\be
\label{eq:phidef}
\varphi_k(X) \defeq \prod_{\substack{E_k(z) = 0\\ j(z)\neq 0, 1728}} (X-j(z)).
\ee
For example, $\varphi_{16}(X) = X-\frac{3456000}{3617}$ and 
$\varphi_{24}(X) = X^2-\frac{340364160000 }{236364091}X+\frac{30710845440000}{236364091}$.
Gekeler observed that the polynomials $\varphi_k$ seem to be irreducible, and  verified this claim for  $k\leq446$ (\cite{gekeler}). 
Here we show that 
the polynomials $\varphi_k$
are irreducible for infinitely many $k$.
More precisely, we have the following theorem.
\begin{theorem}
\label{thm:main}
Let $\varphi_k$ be as in \eqref{eq:phidef}. 
Then $\varphi_{12\cdot 2^{\ell}}$ is irreducible for any 
$\ell \in\mathbb{N}\cup\{0\}$.
\end{theorem}
   Our method employs Dumas' irreducibility criterion and 
   a recurrence for the Eisenstein series $G_k$ due to Popa \cite{popa}.

\section{Preliminaries}
Let 
$G_k \defeq 2 \zeta(k) E_k$ where $\zeta(k) \defeq \sum_{n=1}^{\infty} \frac{1}{n^k}$.
For $k\geq 2$ write
\be
\label{eq:w}
G_k= \sum_{4a+6b=k} w_{a,k} G_4^{a} G_6^{b},
\ee
where $w_{a,k}\in\mathbb{Q}$.
We have the following recurrence for $G_k$.
\begin{proposition}
\label{prop:popa}
We have
\begin{multline}
\label{eq:recurrencePopa}
c_k d_k G_k  = 
\sum_{\substack{j=3\\ j \text{ odd}}}^{k/2-2} \(\binom{k/2}{j} + \binom{k/2-2}{j} \)d_{j+1}d_{k-j-1} G_{j+1} G_{k-j-1}\\
+(k-2) d_2 d_{k-2} G_{2}G_{k-2} + \frac{\pi^2 d_{k-2}}{2} G_{k-2}' 
+\frac{k}{2}d_{k/2}^2 G_{k/2}^2,
\end{multline}
where 
$$f' \defeq \frac{1}{2\pi i} \frac{df}{dz},$$
\be
\label{eq:ck}
c_k =\frac{k}{2(k/2+1)(k/2-1)} + (-1)^{k/2} \frac{(k/2)! (k/2-2)!}{2(k-1)!},
\ee
and
\be
\label{eq:dk}
d_k = \frac{(-1)^{k/2} (k-1)!}{2^{k+1}}.
\ee
\end{proposition}
\begin{proof}
This follows after some computation from equation (A.3) in \cite{popa} with $m=k/2$, $\tilde{m}=k/2-2$, and $w=k-2$.
\end{proof}
{\it Remark}. The Eisenstein series $\widetilde{G_k}$ in \cite{popa} are normalized so that the coefficient of $q$ in their Fourier expansion equals $1$. This differs from our $G_k$ by a factor of 
$\pi^k/d_k$, so that 
$G_k = \frac{\pi^k}{d_k} \widetilde{G_k}$. The identity in Proposition \ref{prop:popa} is
equivalent (via Rankin’s identity) to the Eichler-Shimura relations for odd periods of cusp forms and $G_k$.

The following proposition allows us to relate 
$w_{a,k}$ to the coefficients of $\varphi_{k}$.
\begin{proposition}[\cite{gekeler}, (1.17.1)]
\label{prop:ekdelta}
For $k\equiv 0 \pmod{12}$ we  have 
\bee
2 \zeta(k) \varphi_k(j) = \frac{G_k}{\Delta^{\frac{k}{12}}}.
\eee
\end{proposition}
Let $\nu_{p}(m)$ be the highest power of the prime $p$ that divides a non-zero integer $m$.
Extend this definition to $\mathbb{Q}$ by defining  $\nu_{p}(a/b)\defeq \nu_p(a)-\nu_p(b)$ and
$\nu_p(0) = \infty$.
We will use the following criterion of Dumas.
\begin{proposition}[{\cite[Corollary 1.3]{dumasgen}}]
\label{prop:dumas}
Let $F(x) =  x^n + a_{n-1}x^{n-1} + \cdots + a_0$ be a monic polynomial with coefficients in $\mathbb{Q}$.
If there exists a prime $p$  
satisfying 
\begin{enumerate}[label=(\roman*)]
\item $\frac{\nu_p(a_r)}{n -r} \geq \frac{\nu_p(a_0)}{n}$ for $0\leq r\leq n-1$, and
\item $(\nu_p(a_0), n) = 1$,
\end{enumerate}
then $F(x)$ is irreducible over $\mathbb{Q}$. 
\end{proposition}

\section{Divisibility results}
In this section we prove results about the divisibility of certain binomial coefficients which will be needed in the proof of Theorem \ref{thm:main}.

\begin{lemma}
\label{lem:valsum}
For any even $k\geq 4$ we have
 $$
\nu_{2}\((-1)^{k/2} +\binom{k}{k/2-1}\) = \begin{cases}
1  &\quad \text{ if } k+2 = 2^\ell \text{ for some } \ell\geq3 ,\\
0  &\quad \text{ otherwise.}
\end{cases}
$$
\end{lemma}
\begin{proof}
By Lucas' theorem
(see, for example, \cite[Thm. 1]{FineBinomial})
 we have
 $$
\binom{k}{k/2-1}  \equiv  \begin{cases}
1 \pmod{2} &\quad \text{ if } k+2 = 2^\ell \text{ for some } \ell\geq 3 ,\\
0 \pmod{2} &\quad \text{ otherwise.}
\end{cases}
$$
When
 $k+2 = 2^\ell$ with $\ell\geq 3$
we see from \cite[Theorem 1]{granville} that
 $\binom{k}{k/2-1}\equiv 3 \pmod{4}$. Then,
 $(-1)^{k/2} +\binom{k}{k/2-1} \equiv 2 \pmod{4}$
  and the result follows. 
\end{proof}

\begin{lemma}
\label{lem:ineq}
Let $c_k$ and $d_k$ be as in \eqref{eq:ck} and \eqref{eq:dk}, respectively.
Then, for any even $k\geq 4$ and any $j\leq k/2-2$ the following inequalities hold:
\be
\label{eq:greater1}
\nu_{2}\(\frac{d_{k-2}}{2c_kd_k}\) \geq 1,
\ee
\be
\label{eq:greater0}
\nu_{2}\(\frac{kd_{k/2}^2}{2c_kd_k}\) \geq 0,
\ee
\be
\label{eq:sumineq}
\nu_{2}\( \(\binom{k/2}{j} + \binom{k/2-2}{j} \) \frac{d_{j+1}d_{k-j-1}}{c_kd_k} \) \geq 0.
\ee
\end{lemma}
\begin{proof}
Recall that for any $m\in\mathbb{N}$ we have
\be
\label{eq:factorialv}
\nu_{2}(m!) = m-s_2(m),
\ee
where $s_2(m)$ is the sum of the digits of $m$ in base $2$.
We begin by rewriting
$\frac{d_{k-2}}{2c_kd_k}$ as
$$
\frac{d_{k-2}}{2c_kd_k}
= 
\frac{-2(k-2)!}{(k/2-1)!(k/2)! \((-1)^{k/2} +\binom{k}{k/2-1}\)}.
$$
Using \eqref{eq:factorialv} and Lemma \ref{lem:valsum}
we see that 
\begin{align*}
\nu_{2}\(\frac{d_{k-2}}{2c_kd_k}\)
&= 1+\nu_{2}\((k-2)!\)-\nu_{2}\((k/2-1)!\)-\nu_{2}\((k/2)!\)
-\nu_{2}\((-1)^{k/2} + \binom{k}{k/2-1}\)\\
&= s_2(k/2)
-\nu_{2}\((-1)^{k/2} + \binom{k}{k/2-1}\)\\
&=\begin{cases}
s_2(k/2)-1 &\quad \text{ if } k+2 = 2^\ell, \\
s_2(k/2) &\quad \text{ otherwise.}
\end{cases}
\end{align*}
Clearly $s_2(k/2)\geq 1$, so we only need to consider the case when
$k+2 = 2^\ell$.
Since in this case $s_2(k/2 )= s_2(2^{\ell-1}-1) \geq 2$, we conclude that
$\nu_{2}\(\frac{d_{k-2}}{2c_kd_k}\)  \geq 1$, proving \eqref{eq:greater1}.

To show \eqref{eq:greater0}, we start by rewriting
$\frac{kd_{k/2}^2}{2c_kd_k}$
as
$$
\frac{kd_{k/2}^2}{2c_kd_k}
=
\frac{(k/2-1) }{(-1)^{k/2} +\binom{k}{k/2-1}}.
$$

From this we see that 
\begin{align*}
\nu_{2}\(\frac{kd_{k/2}^2}{2c_kd_k}\) 
&= \nu_2(k/2-1) - \nu_2\((-1)^{k/2} +\binom{k}{k/2-1}\)\\
&=\begin{cases}
 \nu_2(k/2-1)-1 &\quad \text{ if } k+2 = 2^\ell, \\
 \nu_2(k/2-1) &\quad \text{ otherwise.}
\end{cases}
\end{align*}
When
$k+2 = 2^\ell$ we have $\nu_2(k/2 -1)= \nu_2(2^{\ell-1}-2) = 1$. Thus, 
$\nu_{2}\(\frac{d_{k-2}}{2c_kd_k}\)  \geq 0$.

Finally, to prove \eqref{eq:sumineq} we work
with
$\binom{k/2}{j}  \frac{d_{j+1}d_{k-j-1}}{c_kd_k} $
and
$ \binom{k/2-2}{j} \frac{d_{j+1}d_{k-j-1}}{c_kd_k} $
separately.
First, rewrite 
$ \binom{k/2}{j}  \frac{d_{j+1}d_{k-j-1}}{c_kd_k}$ as
$$
 \binom{k/2}{j}  \frac{d_{j+1}d_{k-j-1}}{c_kd_k} 
=\frac{\binom{k-j-2}{k/2-2}}{(-1)^{k/2}+\binom{k}{k/2-1}}.
$$
Then,
\begin{align*}
\nu_{2}\(\binom{k/2}{j}  \frac{d_{j+1}d_{k-j-1}}{c_kd_k}\)
&=
\nu_{2}\(\binom{k-j-2}{k/2-2}\) -\nu_{2}\((-1)^{k/2}+\binom{k}{k/2-1}\)\\
&=\begin{cases}
\nu_{2}\(\binom{k-j-2}{k/2-2}\) -1 &\quad \text{ if } k+2 = 2^\ell, \\
\nu_{2}\(\binom{k-j-2}{k/2-2}\)  &\quad \text{ otherwise.}
\end{cases}
\end{align*}
When
$k+2 = 2^\ell$ we can use Lucas' theorem to obtain
$$
\binom{k-j-2}{k/2-2} = \binom{2^{\ell}-j-4}{2^{\ell-1}-3} \equiv 
\begin{cases}
0 \pmod{2} &\quad \text{ if } j< k/2-2,\\
1 \pmod{2} &\quad \text{ if } j= k/2-2.
\end{cases}
$$
Therefore, when
$k+2 = 2^\ell$ and $j< k/2-2$ we have
$
\nu_{2}\(\binom{k/2}{j}  \frac{d_{j+1}d_{k-j-1}}{c_kd_k}\) \geq 0$ and
when
$j=k/2-2$ we have 
$\nu_{2}\(\binom{k/2}{j}  \frac{d_{j+1}d_{k-j-1}}{c_kd_k}\) = -1$. We will see below that 
in the case
$j=k/2-2$ we have $\nu_{2}\( \binom{k/2-2}{j}  \frac{d_{j+1}d_{k-j-1}}{c_kd_k} \) = -1$, so that
$\nu_{2}\( \(\binom{k/2}{j} + \binom{k/2-2}{j} \) \frac{d_{j+1}d_{k-j-1}}{c_kd_k} \) \geq 0$.

Now we calculate
$\nu_{2}\( \binom{k/2-2}{j}  \frac{d_{j+1}d_{k-j-1}}{c_kd_k} \)$.
We start by rewriting 
$ \binom{k/2-2}{j}  \frac{d_{j+1}d_{k-j-1}}{c_kd_k}$ as
$$
 \binom{k/2-2}{j}  \frac{d_{j+1}d_{k-j-1}}{c_kd_k} 
=\frac{\binom{k-j-2}{k/2}}{(-1)^{k/2}+\binom{k}{k/2-1}}.
$$
Thus,
\begin{align*}
\nu_{2}\(\binom{k/2-2}{j}  \frac{d_{j+1}d_{k-j-1}}{c_kd_k}\)
&=
\nu_{2}\(\binom{k-j-2}{k/2}\) -\nu_{2}\((-1)^{k/2}+\binom{k}{k/2-1}\)\\
&=\begin{cases}
\nu_{2}\(\binom{k-j-2}{k/2}\) -1 &\quad \text{ if } k+2 = 2^\ell, \\
\nu_{2}\(\binom{k-j-2}{k/2}\)  &\quad \text{ otherwise.}
\end{cases}
\end{align*}
When
$k+2 = 2^\ell$ we have
$$
\binom{k-j-2}{k/2} = \binom{2^{\ell}-j-4}{2^{\ell-1}-1} \equiv 
\begin{cases}
0 \pmod{2} &\quad \text{ if } j< k/2-2,\\
1 \pmod{2} &\quad \text{ if } j= k/2-2.
\end{cases}
$$
Therefore, when
$k+2 = 2^\ell$ and $j< k/2-2$ we have
$
\nu_{2}\(\binom{k/2-2}{j}  \frac{d_{j+1}d_{k-j-1}}{c_kd_k}\) \geq 0$ and
when
$j=k/2-2$ we have 
$\nu_{2}\(\binom{k/2}{j}  \frac{d_{j+1}d_{k-j-1}}{c_kd_k}\) = -1$.
\end{proof}

\section{Proof of Theorem \ref{thm:main}}
We start by showing that $\min_{a}\(\nu_2(w_{a,k})\) \geq 0$ for all even $k$.
\begin{theorem}
\label{thm:minGk}
Let $w_{a,k}$ be as in \eqref{eq:w} and let
$k\geq 4$ be an even integer.
Then $\min_{a}\(\nu_2(w_{a,k})\) \geq 0$.
\end{theorem}
\begin{proof}
By induction. The base case is clear.
Suppose that r is even and that $\min_{a}\(\nu_2(w_{a,k})\) \geq 0$
for all even  $k \leq r-2$.
We will show that $\min_{a}\(\nu_2(w_{a,r})\) \geq 0$.
From  \eqref{eq:recurrencePopa}
we have
\begin{multline*}
 G_r  = 
\sum_{\substack{j=3\\ j \text{ odd}}}^{r/2-2} \(\binom{r/2}{j} + \binom{r/2-2}{j} \)\frac{d_{j+1}d_{r-j-1}}{c_r d_r} G_{j+1} G_{r-j-1}\\
+(r-2) \frac{d_2 d_{r-2}}{c_r d_r} G_{2}G_{r-2} + \frac{\pi^2 d_{r-2}}{2c_r d_r} G_{r-2}' 
+\frac{rd_{r/2}^2}{2c_r d_r} G_{r/2}^2.
\end{multline*}
Note that
$\pi^2 G_{4}' = G_2G_4 - \frac{7}{2}G_6$
and 
$\pi^2 G_{6}' = \frac{3}{2} G_2G_6 - \frac{15}{7} G_{4}^2$.
From \eqref{eq:w} we see that
$G_{r-2}= \sum_{4a+6b=r-2} w_{a,r-2} G_4^{a} G_6^{b}$. 
Since 
$(G_4^{a} G_6^{b})' =a G_4^{a-1} G_6^b G_4' + b G_4^aG_6^{b-1} G_6'$
and
$d_2 = -1/8$
we have 
\begin{align*}
&(r-2) \frac{d_2 d_{r-2}}{c_r d_r} G_{2}G_{r-2} + \frac{\pi^2 d_{r-2}}{2c_r d_r} G_{r-2}'\\
&=
\frac{d_{r-2}}{2c_rd_r}  \sum_{4a+6b=r-2} w_{a,r-2} 
\(
G_2G_4^aG_6^b\(2d_2(r-2) + a+\frac{3b}{2}\)
-\frac{7a}{2} G_4^{a-1} G_{6}^{b+1} -\frac{15b}{7} G_4^{a+2} G_6^{b-1}
\)\\ 
&=-\frac{d_{r-2}}{2c_rd_r}  \sum_{4a+6b=r-2} w_{a,r-2} 
\(
\frac{7a}{2} G_4^{a-1} G_{6}^{b+1} + \frac{15b}{7} G_4^{a+2} G_6^{b-1}
\).
\end{align*}
We also see from  \eqref{eq:w}  that
\bee
G_{j+1} G_{r-j-1} = 
\(\sum_{4a+6b=j+1} w_{a,j+1} G_4^{a} G_6^{b}\)
\(\sum_{4a+6b=r-j-1} w_{a,r-j-1} G_4^{a} G_6^{b}\).
\eee
Using Lemma \ref{lem:ineq} and the induction hypothesis
we conclude that $\min_{a}\( \nu_2(w_{a,r})\)  \geq 0$.
\end{proof}
We have the following precise conjecture about 
the value of $\min_{a}\(\nu_2(w_{a,k})\)$.
\begin{conjecture}
Let $w_{a,k}$ be as in \eqref{eq:w} and let $k\geq 4$ be an even integer.
Then,
 $$\min_{a}\(\nu_2(w_{a,k})\)= \begin{cases}
s(k)-2 &\text{ if } k\neq 2^j,\\
0 &\text{ if } k=2^j.
 \end{cases}$$
 \end{conjecture} 
 This has been verified for $k\leq 3500$.

\begin{lemma}
\label{lem:ckidki}
Let $k\equiv 0\pmod{12}$ and let $r\geq0$.
Write
\be
\label{eq:phirep}
 \varphi_k(X) = X^{\frac{k}{12}} + t_{k,\frac{k}{12}-1}X^{\frac{k}{12}-1} + \cdots + t_{k,1}X+t_{k,0}.
\ee
Then,
\bee
t_{k,r} =
\frac{\pi^k}{\zeta(k)}  (-1)^{k/12-r}
\sum_{a=0}^{r}    w_{3a,k} \frac{ 2^{2k/3-6r-2a-1}   }  
{    3^{k/4+3r} 5^{a+k/6} 7^{k/6-2a}    } 
\binom{k/12-a}{k/12-r}.
\eee
\end{lemma}
\begin{proof}
We have
\begin{align*}
\frac{E_{k}}{\Delta^{\frac{k}{12}}} = \frac{G_{k}}{2 \zeta(k) \Delta^{\frac{k}{12}}}
&=  \frac{\sum_{a=0}^{k/12} w_{3a,k}  G_4^{3a} G_6^{k/6-2a}}
{2\zeta(k)\Delta^{k/12}} \\
&=  \frac{\pi^k}{3^{k/2} \zeta(k)} \sum_{a=0}^{k/12} 
   \frac{  w_{3a,k}\cdot  2^{k/6-2a-1} }{5^{k/6+a} 7^{k/6-2a} }  
   j^a (j-1728)^{k/12-a}\\
  &=  \frac{\pi^k }{3^{k/2}\zeta(k)}
  \sum_{a=0}^{k/12} \sum_{i=0}^{k/12-a} 
   \frac{  w_{3a,k}\cdot  2^{k/6-2a-1} }{5^{k/6+a} 7^{k/6-2a} }
  \binom{k/12-a}{i}j^{k/12-i} (-1728)^i.
  \end{align*}
 From Proposition \ref{prop:ekdelta} we see that
 \bee
\frac{E_{k}}{\Delta^{\frac{k}{12}}} = \varphi_k(j)=
j^{\frac{k}{12}} 
+ t_{k,\frac{k}{12}-1}j^{\frac{k}{12}-1} + \cdots + t_{k,1}j+t_{k,0}.
\eee
Comparing coefficients, the result follows.
\end{proof}

\begin{lemma}
\label{lem:dk0andr}
Let  $k= 12\cdot 2^{\ell}$ with $\ell\geq 0$ and let
$w_{a,k}$ be as in \eqref{eq:w}.
Then
$v_2(w_{0,k} ) = 0$ and
$v_2(w_{3a,k} ) \geq 1$ 
 for $ 1\leq a \leq \frac{k}{12}-1$.
\end{lemma}
\begin{proof}
By induction on $\ell$.
From \cite[(59.6)]{rademacherbook}, we have
\be
\label{eq:recurrence}
(k/2-3)(k-1)(k+1) G_{k}=
3\sum_{p=2}^{k/2-2}  (2p-1)(k-2p-1) G_{2p}G_{k-2p}.
\ee
 Note that 
 $G_{12} = \frac{25}{143} G_6^2 + \frac{18}{143} G_4^3$,
 so the result holds for $\ell = 0$.
Suppose that for some $m\geq 0$ we have
$\nu_2(w_{0,12\cdot 2^{m}})=0$ and $\nu_2(w_{3a,12\cdot 2^{m}})\geq 1$ for 
$1\leq a \leq \frac{k}{12}-1$.
Then, 
paring the terms arising from $p$ and $12\cdot 2^m-p$ in 
\eqref{eq:recurrence} gives
\begin{multline*}
G_{12\cdot 2^{m+1}} = \frac{3}{(12\cdot 2^m - 3)(12\cdot 2^{m+1}-1)
(12\cdot 2^{m+1}+1)} 
\\
\(2\sum_{p=2}^{12\cdot 2^{m-1}-1}
(2p-1)(12\cdot 2^{m+1}-2p-1) G_{2p} G_{12\cdot 2^{m+1}-2p} + (12\cdot 2^m-1)^2 G_{12\cdot 2^m}^2\).
\end{multline*}
Thus, using Theorem \ref{thm:minGk}, we have $\nu_2(w_{0,12\cdot 2^{m+1}})=0$
and  
$\nu_2(w_{3a,12\cdot 2^{m+1}})\geq 1$
 for $ 1\leq a \leq \frac{k}{12}-1$.
\end{proof}

\begin{corollary}
\label{cor:v2t}
Let $k= 12\cdot 2^{\ell}$ with $\ell \geq 0$ and let $t_{k,r}$ be as in \eqref{eq:phirep}.
Then,
  $\nu_{2}(t_{k,0}) = \frac{2k-3}{3}$ 
and for $ 1\leq r \leq k/12-1$ we have
 $\nu_{2}(t_{k,r}) \geq \frac{2k}{3}-8r$.
 \end{corollary}
\begin{proof}
Recall (see, for example, \cite[(5.5), (9.1)]{rademacherbook}) that 
 \bee
\nu_2(B_{2n})=-1,
\eee 
and that
\bee
\zeta(2n) = (-1)^{n-1} \frac{(2\pi)^{2n} B_{2n}}{2(2n)!}.
\eee
Thus,
 $\nu_2\( \frac{\zeta(k)}{\pi^{k}}\) =0$.
The result  now follows from Lemmas \ref{lem:dk0andr} and \ref{lem:ckidki}.
 \end{proof}
 We can now prove the main theorem.

 {\it Proof of Theorem \ref{thm:main}}:
 Let $t_{k,r}$ be as in \eqref{eq:phirep}.
For 
$0\leq r\leq \frac{k}{12}-1$, Corollary \ref{cor:v2t} gives
\be
 \frac{\nu_2(t_{k,0})}{k/12} \leq \frac{\nu_2(t_{k,r})}{k/12-r}
 \ee
and
$$\(v_2(t_{k,0}), \frac{k}{12}\) = (2k/3-1, 2^{\ell})=(2^{2 \ell+1} \cdot 3^{\ell-1} -1, 2^{\ell})= 1.$$
Thus,
the result follows from
Proposition \ref{prop:dumas}
with $p= 2$.

 \noindent {\bf Acknowledgements.}
The author thanks Scott Ahlgren for his insightful suggestions.
The author was partially supported by
the Alfred P. Sloan Foundation's MPHD Program, awarded in 2017.
\bibliographystyle{alpha} 
\bibliography{poly}

\end{document}